\documentclass[a4paper,10pt]{amsart}

\usepackage{amssymb}
\usepackage{amsmath}
\usepackage{amsthm}
\usepackage{color}
\usepackage{enumerate}
\newcommand{\Hom}{\operatorname{Hom}}

\newcommand{\Ext}{\operatorname{Ext}}
\newcommand{\Tor}{\operatorname{Tor}}

\newcommand{\Gen}{\operatorname{Gen}}
\newcommand{\Ker}{\operatorname{Ker}}

\newcommand{\Img}{\operatorname{Im}}

\newcommand{\X}{\mathcal{X}}

\newcommand{\Flat}{\mathrm{Flat}}

\newcommand{\Proj}{\mathrm{Proj}}

\DeclareMathOperator{\TFree}{TFree}

\newcommand{\ModR}{\mathrm{Mod}\textrm{-}R}
\newcommand{\RMod}{R\textrm{-}\mathrm{Mod}}

\newcommand{\Filt}[1]{\operatorname{Filt}{#1}}

\newcommand{\Prod}{\mathrm{Prod}}

\DeclareMathOperator{\FPD}{FPD}
\DeclareMathOperator{\pd}{pd}

\DeclareMathOperator{\ML}{ML}

\theoremstyle{plain}
\newtheorem{thm}{Theorem}[section]
\newtheorem{lem}[thm]{Lemma}
\newtheorem{prop}[thm]{Proposition}
\newtheorem{cor}[thm]{Corollary}

\theoremstyle{definition}
\newtheorem{defn}[thm]{Definition}

\theoremstyle{remark}
\newtheorem{rem}[thm]{Remark}

\newtheorem{expls}[thm]{Examples}

\usepackage{tikz}
\usepackage{tikz-cd}

\title{TOR-PAIRS: PRODUCTS AND APPROXIMATIONS}

\author{Manuel Cort\'es-Izurdiaga}
\address{Department of Mathematics, University of Almeria, E-04071, Almeria, Spain}
\email{mizurdia@ual.es}

\begin{document}

\maketitle

\begin{abstract}
  Recently the author has studied rings for which products of flat
  modules have finite flat dimension. In this paper we extend the
  theory to characterize when products of modules in $\mathcal T$ have
  finite $\mathcal T$-projective dimension, where $\mathcal T$ is the
  left hand class of a Tor-pair $(\mathcal T,\mathcal S)$, relating
  this property with the relative $\mathcal T$-Mittag-Leffler
  dimension of modules in $\mathcal S$. We apply
  these results to study the existence of approximations by modules in
  $\mathcal T$. In order to do this, we give short proofs of the well
  known results that a deconstructible class is precovering and that a
  deconstructible class closed under products is preenveloping.
\end{abstract}

\section*{Introduction}
\label{sec:preliminaries}

Let $R$ be an associative ring with unit. A Tor-pair over $R$ is a
pair of classes $(\mathcal T,\mathcal S)$ of right and left
$R$-modules respectively, which are mutually $\Tor$-orthogonal (see
Section 1 for details). The main objective of this work is to study
two problems about
Tor-pairs over $R$: when $\mathcal T$ is closed under
products and when $\mathcal T$ provides for approximations.

The study of when $\mathcal T$ is closed under products is related
with right Gorenstein regular rings. The ring $R$ is said to be right
Gorenstein regular \cite[Definition 2.1]{EnochsCortesTorrecillas} if
the category of right $R$-modules is a Gorenstein cateogry in the
sense of \cite[Definition 2.18]{EnochsEstradaGarciaRozas}. These rings
may be considered as the natural one-sided extension of classical
Iwanaga-Gorenstein rings to non-noetherian rings (recall that the ring
$R$ is Iwanaga-Gorenstein if it is two sided noetherian with finite
left and right self-injective dimension).

In \cite[Corollary VII.2.6]{BeligiannisReiten} it is proved that the
ring $R$ is right Gorenstein regular if and only if the class of all
right $R$-modules with finite projective dimension coincides with the
class of all right modules with finite injective dimension. If we look
at the class $\Proj_\omega$ of all modules with finite projective
dimension, this condition has two consequences: the right projective
finitistic dimension of $R$ is finite (that is,
$\Proj_\omega = \Proj_n$ for some natural number $n$, where $\Proj_n$
denotes the class of all modules with projective dimension less than
or equal to $n$); and the class $\Proj_\omega$ is closed under
products. As in the classical case of products of projective modules
studied in \cite{Chase}, this last property implies that products of
modules with finite flat dimension have finite flat
dimension. Consequently, the first step in order to understand right
Gorenstein regular rings is to study rings with this property. This
study is developed in \cite{Izurdiaga}.

In the first part of this paper we extend the theory of
\cite{Izurdiaga} to characterize, for a fixed Tor-pair
$(\mathcal T,\mathcal S)$, when products of modules in $\mathcal T$
have finite $\mathcal T$-projective dimension (see Definition
\ref{d:RelativeDimension} for the definition of relative dimensions). As in the case of the flat modules, this
property is related with the $\ML(\mathcal T)$-projective dimension of
modules in $\mathcal S$, see Theorem \ref{t:MainTheorem} (where
$\ML(\mathcal T)$ is the class of all Mittag-Leffler modules with
respect to $\mathcal T$, see Definition \ref{d:RelativeML}).

In the second part of the paper we are interested in approximations by
modules in $\mathcal T$ and in $\mathcal T_n$ (modules with
$\mathcal T$-projective dimension less than or equal to $n$). The relationship of
these approximations with the first part of the paper comes from the
fact that if a class of right $R$-modules is preenveloping then it is
closed under products \cite[Propostion 1.2]{HolmJorgensen}. So that, a
natural question arises: if $\mathcal T_n$ is closed under products,
when is it preenveloping?

One tool in order to construct approximations of modules is that of
deconstruction of classes, because a deconstructible class is always
precovering, \cite[Theorem 2.14]{SaorinStovicek} and \cite[Theorem
5.5]{Enochs12}, and a deconstructible class closed under direct
products is preenveloping \cite[Theorem
4.19]{SaorinStovicek}). The procedure of deconstruction of a class
$\mathcal X$ consists on finding a set $\mathcal S$ such that each
module in $\mathcal X$ is $\mathcal S$-filtered, which means that for
each $X \in \mathcal X$ there exists a continuous chain of submodules
of $X$, $\{X_\alpha: \alpha < \kappa\}$ (where $\kappa$ is a
cardinal), whose union is $X$ and such that
$\frac{X_{\alpha+1}}{X_\alpha} \in \mathcal S$.

In Section 3 we give easy proofs of \cite[Theorem
2.14]{SaorinStovicek} and \cite[Theorem 5.5]{Enochs12} (in Theorem
\ref{t:DeconstructivePrecovering}) and of \cite[Theorem
4.19]{SaorinStovicek} (in Theorem
\ref{t:DeconstructiblePreenveloping}), and we prove that $\mathcal T_m$
is deconstructible for each natural number $m$, so that it is always
precovering and it is preenveloping precisely when it is closed under
products (see Corollary \ref{c:TApproximations}).

Throught the paper $R$ will be an associative ring with unit. We shall
denote by $\ModR$ and $\RMod$ the categories of all right $R$-modules
and left $R$-modules respectively. Given a class $\mathcal X$ of right
$R$-modules, we shall denote by $\Prod(\mathcal X)$ the class
consisting of all modules isomorphic to a direct products of modules
in $\mathcal X$. The classes of flat and projective
modules will be denoted by $\Flat_R$ and $\Proj_R$ respectively. If
there is no possible confussion, we shall omit the subscript $R$. The
cardinal of a set $X$ will be denoted by $|X|$.

\section{Tor-pairs, relative dimensions and relative Mittag-Leffler modules}
\label{sec:preliminaries-1}

Given a class $\mathcal X$ of right (resp. left) $R$-modules we shall
denote by $\mathcal X^{\top}$ (resp ${^\top}\mathcal X$) the class of
all left (resp. right) $R$-modules $M$ satisfying $\Tor_1^R(X,M)=0$
(resp $\Tor_1^R(M,X)=0$) for each $X \in \mathcal X$. Recall that a
Tor-pair is a pair of classes $(\mathcal T, \mathcal S)$ such that
$\mathcal T = {^\top}\mathcal S$ and $\mathcal S = \mathcal
T^\top$. Given a class $\mathcal X$ of right $R$-modules (resp. left
$R$-modules), the pair $({^\top}(\mathcal X^\top),\mathcal X^\top)$
(resp. $({^\top}\mathcal X,({^\top}\mathcal X)^\top)$) is a Tor-pair,
which is called the Tor-pair generated by $\mathcal X$.

Given a class $\mathcal X$ of left modules, a short exact sequence of
right modules
\begin{displaymath}
  \begin{tikzcd}
    0 \arrow{r} & A \arrow{r}{f} & B \arrow{r}{g} & C \arrow{r} & 0
  \end{tikzcd}
\end{displaymath}
is called $\mathcal X$-pure if the sequence
\begin{displaymath}
  \begin{tikzcd}
    0 \arrow{r} & A\otimes X \arrow{r} & B\otimes X \arrow{r} &
    C\otimes X \arrow{r} & 0
  \end{tikzcd}
\end{displaymath}
is exact for each $X \in \mathcal X$. In such case, $f$ is called an
$\mathcal X$-pure monomorphism and $g$ an $\mathcal X$-pure
epimorphism. Note that each pure exact sequence is $\mathcal X$-pure
exact. 

\begin{prop}
  Let $(\mathcal T,\mathcal S)$ be a Tor-pair. Then $\mathcal T$ is
  closed under direct limits, pure submodules and $\mathcal S$-pure
  quotients.
\end{prop}

\begin{proof}
  $\mathcal T$ is closed under direct limits since the $\Tor$ functor
  commutes with direct limits. $\mathcal T$ is closed under pure
  submodules by \cite[Proposition 9.12]{AngeleriHerbera}. In order to
  see that it is closed under $\mathcal S$-pure quotients, take
  $f:T \rightarrow T'$ a pure epimorphism with $T \in \mathcal T$ and
  denote by $\iota$ the inclusion of $\Ker f$ into $T$. Given
  $S \in \mathcal S$ and applying $-\otimes_RS$ we get the exact
  sequence
  \begin{displaymath}
    \begin{tikzcd}
      \Tor_1^R(T,S) \arrow{r} & \Tor_1(T',S) \arrow{r} & \Ker f
      \otimes_R S \arrow{r}{\iota \otimes S} & T \otimes_RS
    \end{tikzcd}
  \end{displaymath}
  Since $T \in \mathcal T$, the first term is zero and, since $\ker f$
  is a $\mathcal S$-pure submodule of $T$, $\iota \otimes_R S$ is monic. Then
  $\Tor_1^R(T',S)=0$ and, as $S$ is arbitrary, $T'$ belongs to
  $\mathcal T$.
\end{proof}

A class $\mathcal X$ of right $R$-modules is called resolving if it
contains all projective modules and is closed under extensions and
kernels of epimorphisms. A cotorsion pair $(\mathcal F,\mathcal C)$ is
hereditary if $\mathcal F$ is resolving. Similarly, we shall call a
Tor-pair $(\mathcal T, \mathcal S)$ hereditary if $\mathcal T$ is
resolving. The following result is the Tor-pair version of the well
known characterizations of hereditary cotorsion pairs \cite[Theorem
1.2.10]{GarciaRozas}.

\begin{prop}\label{p:Hereditary}
  Let $(\mathcal T,\mathcal S)$ be a Tor-pair. Then:
  \begin{enumerate}
  \item The Tor pair is hereditary.
    
  \item $\mathcal S$ is resolving.

  \item $Tor^R_n(T,S)=0$ for each $T \in \mathcal T$,
    $S \in \mathcal S$ and nonzero natural number $n$.
  \end{enumerate}
\end{prop}

\begin{proof}
  (1) $\Rightarrow$ (3). From (1) follows that all syzygies of any
  module in $\mathcal T$ belong to $\mathcal T$. Then (3) is a
  consequence of \cite[Corollary 6.23]{Rotman}.

  (3) $\Rightarrow$ (1). Given a short exact sequence
  \begin{displaymath}
    \begin{tikzcd}
      0 \arrow{r} & K \arrow{r} & T_1 \arrow{r} & T_2 \arrow{r} & 0
    \end{tikzcd}
  \end{displaymath}
  with $T_1, T_2 \in \mathcal T$, the induced long exact sequence when
  tensoring with any $S \in \mathcal S$ gives an isomorphism
  $\Tor_2^R(T_2,S) \cong \Tor_1^R(K,S)$. Then (3) gives that
  $\Tor_1^R(K,S)=0$ and $K \in \mathcal S$.

  (3) $\Leftrightarrow$ (2). Follows from the previous proof, since
  (3) is left-right symmetric.
\end{proof}

\begin{prop}
  Let $(\mathcal T,\mathcal S)$ be a hereditary Tor-pair. Then
  $\mathcal T$ is closed under $\mathcal S$-pure submodules.
\end{prop}

\begin{proof}
  We can argue as in \cite[Proposition 9.12]{AngeleriHerbera}. Let $T$
  be a module in $\mathcal T$ and $K$ a $\mathcal S$-pure submodule of
  $T$. Let $S$ be any module in $\mathcal S$ and take
  \begin{displaymath}
    \begin{tikzcd}
      0 \arrow{r} & S' \arrow{r} & F \arrow{r} & S \arrow{r} & 0
    \end{tikzcd}
  \end{displaymath}
  a projective presentation of $S$. We can construct the following commutative
  diagram with exact rows:
  \begin{displaymath}
    \begin{tikzcd}
      0 \arrow{r} & \Tor_1^R(K,S) \arrow{r} \arrow{d}{f} & K \otimes_R S' \arrow{r} \arrow{d}{g} & K \otimes_R F \arrow{d}\\
      0 \arrow{r} & \Tor_1^R(T,S) \arrow{r} & T \otimes_R S' \arrow{r} & T \otimes_R F\\
    \end{tikzcd}
  \end{displaymath}
  Now $g$ is monic as the inclusion $K \rightarrow T$ is
  $\mathcal S$-pure and $S' \in \mathcal S$ by Proposition
  \ref{p:Hereditary}. Then $f$ is monic and $\Tor_1^R(K,S)=0$ since
  $\Tor_1^R(T,S)=0$. Because $S$ is arbitrary, we conclude that
  $K \in \mathcal T$.
\end{proof}

\begin{expls}
  \begin{enumerate}
  \item The pair of classes $(\Flat, \RMod)$ is a hereditary Tor-pair.

  \item Recall that a left $R$-module $C$ is cyclically presented
    provided that $C \cong \frac{R}{Rx}$ for some $x \in R$. A right
    module $X$ satisfying $\Tor_1^R(X,C)=0$ for each cyclically
    presented left module $C$ is called \textrm{torsion-free}. We
    shall denote by $\TFree$ the class consisting of all torsion-free
    right modules. Then $(\TFree,\TFree^\top)$ is a Tor-pair.
  \end{enumerate}
\end{expls}

We shall use the homological notation for projective resolutions so
that, for a given a right $R$-module $M$, a projective resolution
of $M$ will be denoted
\begin{displaymath}
  \begin{tikzcd}
    \cdots \arrow{r} & P_1 \arrow{r}{d_1} & P_0 \arrow{r}{d_0} & M
    \arrow{r} & 0
  \end{tikzcd}
\end{displaymath}
Then the $nth$-syzygy of $M$ will be $\Ker d_n$ for each natural
number $n$.

\begin{defn}\label{d:RelativeDimension}
  Let $\mathcal X$ be a class of left $R$-modules containing all
  projective modules.
  \begin{enumerate}
  \item Given a nonzero natural number $n$ and a left $R$-module $M$,
    we shall say that $M$ has projective dimension relative to
    $\mathcal X$ (or $\mathcal X$-projective dimension) less than or
    equal to $n$ (written $\pd_{\mathcal X}(M) \leq n$) if there exists
    a projective resolution of $M$ such that its $(n-1)st$ syzygy
    belongs to $\mathcal X$. We shall denote by $\mathcal X_n$ the
    class of all modules with $\mathcal X$-projective dimension less
    than or equal to $n$ (if $n=0$, $\mathcal X_0$ will be
    $\mathcal X$). Moreover, we shall denote
    $\mathcal X_\omega = \bigcup_{n < \omega}\mathcal X_n$.

  \item Given a left $R$-module $M$ the $\mathcal X$-projective
    dimension of $M$ is
    \begin{displaymath}
      \pd_{\mathcal X}(M) = \min\left(\{m < \omega: M \in \mathcal X_m\}
      \cup \{\omega\}\right)
    \end{displaymath}
  \end{enumerate}
\end{defn}

Note that if $\mathcal X$ is closed under direct summands and finite
direct sums, the $\mathcal X$-projective dimension does not depend on
the chosen projective resolution, since, for each natural number $n$,
any two $n$-sysygies of a module are projectively equivalent by
\cite[Proposition 8.5]{Rotman}.

\begin{defn}
  Let $\mathcal X$ and $\mathcal Y$ be a class of left modules such
  that $\mathcal X$ contains all
  projective modules. The
  $\mathcal X$-projective dimension of $\mathcal Y$ is
  \begin{displaymath}
    \pd_{\mathcal X}(\mathcal Y) =  \sup \{\pd_{\mathcal X}(Y):Y \in
    \mathcal Y\}
  \end{displaymath}
\end{defn}

Note that $\pd_{\Proj}(\ModR)$ is the right global dimension of the
ring and $\pd_{\Proj}(\Proj_{\omega})$ is the right finitistic
projective dimension of the ring, i. e., the supremmun of the
projective dimensions of all modules with finite projective
dimension. For a general class of right modules $\mathcal X$ containing
all projective modules, $\pd_{\mathcal X}(\mathcal X_\omega)$ is
called in \cite{BazzoniCortesEstrada} the left finitistic
$\mathcal X$-projective dimension of $R$ and is denoted there by
$\FPD_\mathcal{X}(R)$.

As it is proved in \cite[Lemma 3.9]{Izurdiaga} using an argument from
\cite[Corollary VII.2.6]{BeligiannisReiten}, when $\mathcal Y$ is
closed under countable direct sums or countable direct products, we only have to
see that each module in $\mathcal Y$ has finite
$\mathcal X$-projective dimension in order to get that
$\pd_{\mathcal X}(\mathcal Y)$ is finite.

\begin{lem} \label{l:FinitisticDimensions} \cite[Lemma 3.9]{Izurdiaga}
  Let $\mathcal X$ and $\mathcal Y$ be classes of left $R$-modules
  such that $\mathcal X$ is closed under direct summands, finite
  direct sums and contains all projecive modules, and $\mathcal Y$ is
  closed under countable direct sums or countable direct
  products. Then the following assertions are equivalent:
  \begin{enumerate}
  \item $\pd_{\mathcal X}(Y)$ is finite for each $Y \in \mathcal Y$.

  \item $\pd_{\mathcal X}(\mathcal Y)$ is finite.
  \end{enumerate}
\end{lem}

For cotorsion pairs and Tor-pairs we can compute relative dimensions
using $\Ext$ and $\Tor$ functors respectively.

\begin{lem}\label{l:DimensionTor}
  Let $n$ be a natural number and $M$ a right $R$-module.
  \begin{enumerate}
  \item If $(\mathcal F, \mathcal C)$ is a cotorsion pair in $\ModR$,
    then $\pd_{\mathcal F}(M)\leq n$ if and only if $\Ext^{n+1}_R(M,C)=0$
    for each $C \in \mathcal C$. Moreover
    \begin{displaymath}
      \pd_{\mathcal F}(M)=\min\{n \leq \omega:\Ext^{n+1}_R(M,C)=0 \quad
      \forall C \in \mathcal C\}
    \end{displaymath}

  \item If $(\mathcal T,\mathcal S)$ is a Tor-pair, then
    $\pd_{\mathcal T}(M) \leq n$ if and only if $\Tor_{n+1}^R(M,S)=0$
    for each $S \in \mathcal S$. Moreover,
    \begin{displaymath}
      \pd_{\mathcal T}(M) = \min\{n \leq \omega:\Tor_{n+1}^R(M,S)=0
      \quad \forall S \in \mathcal S\}.
    \end{displaymath}
  \end{enumerate}
\end{lem}

\begin{proof}
  Both proofs are similar. We shall prove (2). Let $K_{n-1}$ be a $(n-1)th$-syzygy of $M$. Then
  $\pd_{\mathcal T}(M)=n$ if and only if $K_{n-1} \in \mathcal T$ if
  and only if $\Tor_1^R(K_{n-1},S)=0$ for each $S \in \mathcal S$. But
  by \cite[Corollary]{Rotman} this is equivalent to $\Tor^R_{n+1}(M,S)=0$
  for each $S \in \mathcal S$.
\end{proof}

Using this result it is easy to compute the dimension of the third
module in a short exact sequence in the same way as it can be done for
the classical projective dimension \cite[Exercise 8.5]{Rotman}. We
shall use this result later.

\begin{prop}\label{p:DimensionInExactSequences}
  Let $(\mathcal T,\mathcal S)$ be a Tor-pair and
  \begin{displaymath}
    \begin{tikzcd}
      0 \arrow{r} & A \arrow{r} & B \arrow{r} & C \arrow{r} & 0
    \end{tikzcd}
  \end{displaymath}
  a short exact sequence of right modules. Then:
  \begin{enumerate}
  \item If $\pd_{\mathcal T}(B) < \pd_{\mathcal T}(A)$ then $\pd_{\mathcal T}(C) = \pd_{\mathcal T}(A)+1$.

  \item If $\pd_{\mathcal T} (B) > \pd_{\mathcal T}(A)$ then $\pd_{\mathcal T}(C) = \pd_{\mathcal T}(B)$.

    \item If $\pd_{\mathcal T}(B) = \pd_{\mathcal T}(A)$, then $\pd_{\mathcal T}(C) \leq \pd_{\mathcal T}(A)+1$.
  \end{enumerate}
\end{prop}

\begin{proof}
  Given a nonzero natural number $n$ and $S \in \mathcal S$ we have,
  by \cite[Corollary 6.30]{Rotman}, the exact sequence
  \begin{equation}
    \label{eq:Tor}
    \begin{tikzcd}[column sep=tiny]
      \Tor_{n+1}^R(B,S) \arrow{r} & \Tor_{n+1}^R(C,S) \arrow{r} &
      \Tor_n^R(A,S) \arrow{r} & \Tor_{n}^R(B,S) \arrow{r} &
      \Tor_n^R(C,S)
    \end{tikzcd}
  \end{equation}
  Set $n_A = \pd_{\mathcal T}(A)$ and $n_B = \pd_{\mathcal T}(B)$;
  take $S_A,S_B \in \mathcal S$ with $\Tor_{n_A}^R(A,S_A) \neq 0$ and
  $\Tor_{n_B}(B,S_B)\neq 0$.

  If $n_B < n_A$, then the sequence (\ref{eq:Tor}) for $n=n_A+1$ gives
  $\Tor_{n_A+2}(C,S)=0$ for each $S \in \mathcal S$. Morover, for
  $n=n_A$ and $S=S_A$, the sequence (\ref{eq:Tor}) gives
  $\Tor_{n_A+1}(C,S_A) \neq 0$. Consequently
  $\pd_{\mathcal T}(C) = n_A+1$.

  If $n_A < n_B$, then the sequence (\ref{eq:Tor}) for $n=n_B$ gives
  $\Tor_{n_B+1}(C,S)=0$ for all $S \in \mathcal S$. The same sequence
  for $n=n_B$ and $S=S_B$ gives that $\Tor_{n_B}^R(C,S_B) \neq 0$, so
  that $\pd_{\mathcal T}(C)=n_B$.

  Finally, if $n_A = n_B$, the sequence (\ref{eq:Tor}) for $n=n_A+1$
  gives that $\Tor_{n_A+2}^R(C,S)=0$ for each $S \in \mathcal S$, so
  that $\pd_{\mathcal T}(C) \leq n_A+1$.
\end{proof}

Mittag-Leffler modules were introduced by Raynaud and Gruson in their
seminal paper \cite{RaynaudGruson}. We shall work with the following
relativization of the concept, introduced in \cite{Rothmaler}.

\begin{defn}\label{d:RelativeML}
  Let $\mathcal X$ be a class of right $R$-modules and $M$ a left
  $R$-module. We say that $M$ is $\mathcal X$-Mittag-Leffler if for
  any family of modules in $\mathcal X$, $\{X_i:i \in I\}$, the
  canonical morphism from
  $\left(\prod_{i \in I}X_i\right) \otimes_R M$ to
  $\prod_{i \in I}(X_i \otimes_R M)$ is monic.
\end{defn}

We shall denote by $\ML(\mathcal X)$ the class consisting of all
$\mathcal X$-Mittag-Leffler left $R$-modules. We are interested in
Mittag-Leffler modules relative to a Tor-pair.

% The following result,
% which is a particular case of some results of \cite{AngeleriHerbera},
% sums up the main properties of these modules.

% \begin{prop}\label{p:PropertiesML}
%   Let $(\mathcal T,\mathcal S)$ be a Tor-pair.
  % \begin{enumerate}
  % \item The class $\ML(\mathcal T)$ is closed under pure submodules,
  %   pure extensions, direct summands and direct sums.
  % \item If $M \in \ML(\mathcal T)$ and $N \leq M$ is finitely
  %   generated, then $\frac{M}{N} \in \ML(\mathcal T)$.

  % \item The class $\mathcal S \cap \ML(\mathcal T)$ is closed under
  %   filtrations.
  % \end{enumerate}
  
% \end{prop}

\section{Tor pairs closed under products}

As we have mentioned before, in \cite{Izurdiaga} they are
characterized rings for which direct products of flat modules have
finite flat dimension. Let $(\mathcal T,\mathcal S)$ be a hereditary
Tor-pair. In this section we study rings for which direct products of
modules in $\mathcal T$ have finite $\mathcal T$-projective
dimension. The main result relates this property with the
$\ML(\mathcal T)$-projective dimension of the class $\mathcal S$.

\begin{thm}\label{t:MainTheorem}
  The following assertions are equivalent for a hereditary Tor-pair
  $(\mathcal T, \mathcal S)$ and a natural number $n$.
  \begin{enumerate}
  \item Each product of modules in $\mathcal T$ has
    $\mathcal T$-projective dimension less than or equal to $n$.

  \item Each module in $\mathcal S$ has finite
    $\ML(\mathcal T)$-projective dimension less than or equal to
    $n+1$.
  \end{enumerate}
  Consequently:
  \begin{displaymath}
    \pd_{\mathcal T}\left(\Prod (\mathcal T)\right) = n \Leftrightarrow \pd_{\ML(\mathcal T)}(\mathcal S) = n+1
  \end{displaymath}
\end{thm}

\begin{proof}
  Fix $\{T_i:i \in I\}$ a family of modules in $\mathcal T$ and $S$ an
  object of $\mathcal S$. Take a projective resolution of $S$,
  \begin{displaymath}
    \begin{tikzcd}
      \cdots \arrow{r}{d_2} & P_1 \arrow{r}{d_1} & P_0 \arrow{r}{d_0}
      & S \arrow{r} & 0
    \end{tikzcd}
  \end{displaymath}
  and consider the short exact sequence
  \begin{displaymath}
    \begin{tikzcd}
      0 \arrow{r} & K_{n} \arrow{r} & P_{n} \arrow{r}{d_{n}} & K_{n-1}
      \arrow{r} & 0
    \end{tikzcd}
  \end{displaymath}
  where $K_{n-1} = \ker d_{n-1}$ and $K_{-1}=S$ if $n=0$. Tensoring by
  $\prod_{i \in I}T_i$ we can construct the following commutative
  diagram with exact rows:
  \begin{displaymath}
    \begin{tikzcd}[cramped,sep=small]
      & \left(\prod_{i \in I}T_i\right) \otimes_R K_{n} \arrow{r}{f} \arrow{d}{g} & \left(\prod_{i \in I}T_i\right) \otimes P_{n} \arrow{r} \arrow{d}{h} & \left(\prod_{i \in I} T_i \right) \otimes K_{n-1} \arrow{r} \arrow{d} & 0\\
      0 \arrow{r} & \prod_{i \in I}T_i \otimes K_{n} \arrow{r} &
      \prod_{i \in I}T_i \otimes P_{n} \arrow{r} & \prod_{i \in I}T_i
      \otimes K_{n-1} \arrow{r} & 0
    \end{tikzcd}
  \end{displaymath}
  (note that the last row is exact because $K_{n-1} \in \mathcal S$ by
  Proposition \ref{p:Hereditary}). Since $P_{n}$ is Mittag-Leffler,
  $h$ is monic and, consequently, $f$ is monic if and only if $g$ is
  monic. By \cite[Corollary 6.23 and Corollary 6.27]{Rotman} there
  exists a exact sequence
  \begin{displaymath}
    \begin{tikzcd}
      0 \arrow{r} & \Tor_{n+1}^R\left(\prod_{i \in I}T_i,S\right)
      \arrow{r} & \left(\prod_{i \in I}T_i\right) \otimes_R K_{n}
      \arrow{r}{f} & \left(\prod_{i \in I}T_i\right) \otimes P_{n}
    \end{tikzcd}
  \end{displaymath}
  so that $f$ is monic if and only if
  $\Tor_{n+1}^R\left(\prod_{i \in I}T_i,S\right)=0$. The conclusion is
  that for a fixed family $\{T_i:i \in I\}$ in $\mathcal T$ and module
  $S \in \mathcal S$, $g$ is monic if and only if
  $\Tor_{n+1}^R\left(\prod_{i \in I}T_i,S\right)=0$.

  Now using that both $\{T_i:i \in I\}$ and $S$ are arbitrary we get,
  by Lemma \ref{l:DimensionTor}, that all products of modules in
  $\mathcal T$ have $\mathcal T$-projective dimension less than or
  equal to $n$ if and only if each module in $\mathcal S$ has
  $\ML(\mathcal T)$-projective dimension less than or equal to $n+1$.
\end{proof}

As an inmediate consequence we get the characterization of when the left hand class of a
Tor-pair is closed under products.

\begin{cor}\label{c:TorPairClosedUnderProducts}
  The following assertions are equivalent for a hereditary Tor-pair
  $(\mathcal T, \mathcal S)$.
  \begin{enumerate}
  \item $\mathcal T$ is closed under products.

  \item Each module in $\mathcal S$ has $\ML(\mathcal T)$-projective
    dimension less than or equal to $1$.
  \end{enumerate}
\end{cor}

If we apply this result to the Tor-pair induced by the flat modules,
we get the following well known results. Recall that the class of flat
Mittag-Leffler modules is closed under extensions since, if
\begin{displaymath}
  \begin{tikzcd}
    0 \arrow{r} & K \arrow{r} & M \arrow{r} & N \arrow{r} & 0
  \end{tikzcd}
\end{displaymath}
is a short exact sequence in $\RMod$ with $K$ and $N$ flat and Mittag-Leffler,
then $M$ is flat, the sequence is pure and, for each family of right
$R$-modules $\{X_i:i \in I\}$ there exists a commutative diagram
\begin{displaymath}
  \begin{tikzcd}
    & \left(\prod_{i \in I}X_i\right) \otimes_R K \arrow{r} \arrow{d}{f} & \left(\prod_{i \in I}X_i\right) \otimes M \arrow{r} \arrow{d}{g} & \left(\prod_{i \in I} X_i \right) \otimes N \arrow{r} \arrow{d}{h} & 0\\
    0 \arrow{r} & \prod_{i \in I}X_i \otimes K \arrow{r} & \prod_{i
      \in I}X_i \otimes M \arrow{r} & \prod_{i \in I}X_i \otimes N
    \arrow{r} & 0
  \end{tikzcd}
\end{displaymath}
from which follows that $g$ is monic, as $f$ and $h$ are.

\begin{cor}
  \begin{enumerate}
  \item Pure submodules of flat Mittag-Leffler right modules are
    Mittag-Leffler.

  \item $R$ is right coherent if and only if each submodule of a
    projective right module is Mittag-Leffler with respect to the flat
    modules.
  \end{enumerate}
\end{cor}

\begin{proof}
  (1) If we apply the previous result to the Tor-pair
  $(\Flat_R,\RMod)$ we get that each flat right module has
  $\ML(\RMod)$-projective dimension less than or equal to $1$, as
  $\RMod$ is closed under products. Noting that $\Flat_R$ consists of
  all pure quotients of projective modules and that $\ML(\RMod)$ is
  the class of all Mittag-Leffler modules, this is equivalent to all
  pure submodules of projective right modules being (flat)
  Mittag-Leffler modules.

  Now let $M$ be a flat Mittag-Leffler right module and $K$ a pure submodule
  of $M$. Let $f:P \rightarrow \frac{M}{K}$ be an epimorphism with $P$
  projective. Making pullback of $f$ along the projection
  $M \rightarrow \frac{M}{K}$ we get the following commutative diagram
  with exact rows and colummns:
  \begin{displaymath}
    \begin{tikzcd}
      & & 0 \arrow{d} & 0 \arrow{d} & \\
      & & \Ker f \arrow{d} \arrow[equal]{r} & \Ker f \arrow{d} & \\
      0 \arrow{r} & K \arrow{r} \arrow[equal]{d} & Q \arrow{r} \arrow{d} & P \arrow{r}\arrow{d} & 0\\
      0 \arrow{r} & K \arrow{r} & M \arrow{r} \arrow{d} & \frac{M}{K}
      \arrow{r} \arrow{d}
      & 0\\
      & & 0 & 0 &
    \end{tikzcd}
  \end{displaymath}
  Since the first column is pure and $P$ is projective, $\Ker f$ is
  flat Mittag-Leffler by the previous proof. Then, as the class of flat Mittag-Leffler
  modules is closed under extensions, $Q$ is flat Mittag-Leffler as well. But the
  middle row is split, so that $K$ is isomorphic to a direct summand
  of $Q$. Thus, $K$ is flat Mittag-Leffler.
  
  (2) If we consider the Tor-pair $(\ModR,{_R}\Flat)$, we get that $R$
  is right coherent if and only if ${_R}\Flat$ is closed under
  products if and only if (by the left version of Corollary \ref{c:TorPairClosedUnderProducts})
  each right module has projective dimension relative to the
  Mittag-Leffler modules less than or equal to $1$. But this is
  equivalent to each submodule of a projective module being
  Mittag-Leffler with respect to the flat modules.
\end{proof}

Now, what about the class $\mathcal T_m$ where $m$ is a nonzero
natural number? When is it closed under products? The following
result, which extends \cite[Proposition 4.1]{Izurdiaga}, gives the
answer.

\begin{prop}\label{p:FiniteDimensionTorPairs}
  The following assertions are equivalent for a Tor pair
  $(\mathcal T, \mathcal S)$.
  \begin{enumerate}
  \item Each module in $\Prod (\mathcal T)$ has finite
    $\mathcal T$-projective dimension.

  \item $\pd_{\mathcal T}\left(\Prod (\mathcal T)\right)$ is finite.

  \item There exists a natural number $m$ such that each module in
    $\Prod (\mathcal T_m)$ has finite $\mathcal T$-projective
    dimension.

  \item There exists a natural number $m$ such that
    $\pd_{\mathcal T}\left(\Prod (\mathcal T_m)\right)$ is finite.

  \item For any natural number $m$ each module in
    $\Prod (\mathcal T_m)$ has finite $\mathcal T$-projective
    dimension.

  \item For any natural number $m$,
    $\pd_{\mathcal T}\left(\Prod (\mathcal T_m)\right)$ is finite.
  \end{enumerate}
  Moreover, when all these conditions are satisfied then
  \[\pd_{\mathcal T}(\Prod(\mathcal T_{m})) \leq \pd_{\mathcal T}\left(\Prod
      (\mathcal T_{m+1})\right) \leq \pd_{\mathcal T}\left(\Prod
      (\mathcal T)\right)+m+1\] for each natural number $m$. If,
  in addition $\pd_{\mathcal T}(\Prod \mathcal T) = 0$ (that is,
  $\mathcal T$ is closed under products), then $\mathcal T_m$ is
  closed under products for each natural number $m$.
\end{prop}

\begin{proof}
  (1) $\Leftrightarrow$ (2), (3) $\Leftrightarrow$ (4) and (5)
  $\Leftrightarrow$ (6) follow from Lemma
  \ref{l:FinitisticDimensions}.

  (1) $\Leftrightarrow$ (4) and (5) $\Rightarrow$ (1) are trivial.

  (1) $\Rightarrow$ (5) is proved by dimension shifting noting that,
  if the result is true for some natural number $m$ and
  $\{T_i:i \in I\}$ is a family of modules having
  $\mathcal T$-projective dimension less than or equal to $m+1$ then,
  for each $i \in I$ there exists a short exact sequence
  \begin{displaymath}
    \begin{tikzcd}
      0 \arrow{r} & K_i \arrow{r} & P_i \arrow{r} & T_i \arrow{r} & 0
    \end{tikzcd}
  \end{displaymath}
  with $P_i$ projective and $K_i \in \mathcal T_m$. Since the direct
  product is an exact functor, these sequences give the exact sequence
  \begin{equation}
    \label{eq:Product}
    \begin{tikzcd}
      0 \arrow{r} & \prod_{i \in I}K_i \arrow{r} & \prod_{i \in I}P_i
      \arrow{r} & \prod_{i \in I}T_i \arrow{r} & 0
    \end{tikzcd}
  \end{equation}
  in which both the first and second term have finite
  $\mathcal T$-projective dimension by the induction hyphotesis. Then
  so has $\prod_{i \in I}T_i$ by Proposition
  \ref{p:DimensionInExactSequences}.

  In order to prove the last inequality we shall proceed by induction
  on $m$. Suppose that we have proved the result for some natural
  number $m$. The first inequality is trivial, since
  $\mathcal T_m \subseteq \mathcal T_{m+1}$. In order to prove the
  other one simply note that for any family of modules
  $\{T_i:i \in I\}$ in $\mathcal T_{m+1}$ we can construct, as above, a
  short exact sequence
  \begin{displaymath}
    \begin{tikzcd}
      0 \arrow{r} & K \arrow{r} & P \arrow{r} & \prod_{i \in I} T_i
      \arrow{r} & 0
    \end{tikzcd}
  \end{displaymath}
  with $P \in \Prod(\mathcal T)$ and $K \in \Prod(\mathcal
  T_m)$. Using Proposition \ref{p:DimensionInExactSequences} and the
  induction hyphotesis we get the desired inequality.

  Finally, if $\pd_{\mathcal T}(\Prod (\mathcal T))=0$ we induct on
  $m$. If $\mathcal T_{m} = \mathcal T_{m-1}$, the result follows from
  the induction hyphotesis. If $\mathcal T_m \neq \mathcal T_{m-1}$,
  then the preceeding inequality gives
  $\pd_{\mathcal T}(\Prod(\mathcal T_m)) \leq m$. In addition
  $m \leq \pd_{\mathcal T}(\Prod(\mathcal T_m))$ as well, so that
  $\mathcal T_m$ is closed under products.
\end{proof}

As an application of this result we can characterize when the class $\mathcal T_\omega$ is
closed under products:

\begin{cor}
  The following assertions are equivalent for a hereditary Tor-pair
  $(\mathcal T, \mathcal F)$.
  \begin{enumerate}
  \item $\mathcal T_{\omega}$ is closed under direct products.

  \item $\pd_{\mathcal T}(\mathcal T_{\omega})$ and $\pd_{\ML(\mathcal
      T)}(\mathcal S)$ are finite. That is, the right finitistic
    $\mathcal T$-projective dimension is finite and each module in
    $\mathcal S$ has finite $\ML(\mathcal T)$-projective dimension.
  \end{enumerate}
\end{cor}

\begin{proof}
  (1) $\Rightarrow$ (2). If $\mathcal T_\omega$ is closed under direct
  products, we can apply Lemma \ref{l:FinitisticDimensions} to get
  that $\pd_{\mathcal T}(\mathcal T_{\omega})$ is finite. That is,
  $\mathcal T_\omega = \mathcal T_n$ for some natural number $n$. Now
  $\pd_{\ML(\mathcal T)}(\mathcal S)$ is finite as a consequence of
  Theorem \ref{t:MainTheorem} and Proposition \ref{p:FiniteDimensionTorPairs}.

  (2) $\Rightarrow$ (1). Since $\pd_{\mathcal T}(\mathcal T_\omega)$
  is finite, there exists a natural number $n$ such that
  $\mathcal T_\omega = \mathcal T_n$. Now, as $\pd_{\ML(\mathcal
    T)}(\mathcal S)$, apply Corollary
  \ref{c:TorPairClosedUnderProducts} to get that each product of modules in
  $\mathcal T$ has finite $\mathcal T$-projective dimension. By
  Proposition \ref{p:FiniteDimensionTorPairs}, $\mathcal T_n$ is
  closed under products as well.
\end{proof}

Recall that a class $\mathcal X$ of right $R$-modules is definable if
it is closed under direct products, direct limits and pure
submodules. As a consequence of the results of this section we can
characterize when, fixed a Tor-pair $(\mathcal T,\mathcal S)$, the
classes $\mathcal T_m$ are definable for each natural number $m$. The
same proof of \cite[Proposition 4.7]{Izurdiaga} gives:

\begin{prop}\label{p:closure}
  Let $(\mathcal T, \mathcal S)$ be a hereditary Tor-pair and $n$ a
  natural number. Then $\mathcal T_n$ is closed under direct limits
  and pure submodules.
\end{prop}

\begin{proof}
  The closure under direct limits follows from Lemma
  \ref{l:DimensionTor} and the fact that the $\Tor_{n+1}^R$ functor
  commutes with direct colimits. In order to see that $\mathcal T_n$
  is closed under pure submodules take $T \in \mathcal T_n$ and $T'$ a
  pure submodule of $T$. Arguing as in \cite[Proposition
  4.7]{Izurdiaga} we get, for each $S \in \mathcal S$, the exact
  sequence
  \begin{displaymath}
    \begin{tikzcd}
      0 \arrow{r} & \Tor_{n+1}^R(T',S) \arrow{r} & \Tor_{n+1}^R(T,S)
      \arrow{r} & \Tor_{n+1}^R\left(\frac{T}{T'},S\right) \arrow{r} &
      0
    \end{tikzcd}
  \end{displaymath}
  Now, if $T \in \mathcal T_n$ then so does $T'$ by Lemma
  \ref{l:DimensionTor}.
\end{proof}

Putting all things together, we charactize when $\mathcal T_m$ is a
definable class for each $m \in \mathbb N$. This result extends
\cite[Proposition 9.12]{AngeleriHerbera}

\begin{cor}\label{c:Definable}
  The following assertions are equivalent for a
  hereditary Tor-pair $(\mathcal T, \mathcal S)$.
  \begin{enumerate}
  \item Each module in $\mathcal S$ has $\ML(\mathcal T)$-projective
    dimension less than or equal to $1$.
    
  \item $\mathcal T$ is closed under products.

  \item $\mathcal T_m$ is a definable category for each natural number
    $m$.
  \end{enumerate}
\end{cor}

\begin{proof}
  (1) $\Leftrightarrow$ (2) is Corollary
  \ref{c:TorPairClosedUnderProducts}. (2) $\Leftrightarrow$ (3)
  follows from propositions \ref{p:FiniteDimensionTorPairs} and \ref{p:closure}.
\end{proof}

\section{Approximations by modules in $\mathcal T_m$}
\label{sec:appr-modul-mathc-1}

In this section we study the existence of approximations by modules in
$\mathcal T_m$ for each natural number $m$. Let $\mathcal X$ be a
class of right $R$-modules and $M$ a module. A $\mathcal X$-precover
of $M$ is a morphism $f:X \rightarrow M$ with $X \in \mathcal X$ such
that for each $X' \in \mathcal X$, the induced morphism
$\Hom_R(X',X) \rightarrow \Hom_R(X',M)$ is epic. The
$\mathcal X$-precover $f$ is said to be a $\mathcal X$-cover if it is
minimal in the sense that each endomorphism $g$ of $X$ satisfying
$fg=f$ is an isomorphism. The class $\mathcal X$ is called precovering
or covering if each right module has a $\mathcal X$-precover or a $\mathcal X$-cover respectively. Dually are
defined $\mathcal X$-preenvelopes and $\mathcal X$-envelopes, and the
corresponding preenveloping and enveloping classes.

Most of the known examples of classes providing for approximations are
part of a ``small'' cotorsion pair $(\mathcal F, \mathcal C)$ (in the
sense that it is generated by a set, i. e., there exists a set of
modules $\mathcal G$ such that $\mathcal C = \mathcal G^\perp$). This
is due to the fact that a cotorsion pair generated by a set always
provide for precovers and preenvelopes, \cite[Theorem
3.2.1]{GobelTrlifaj} and \cite[Lemma 2.2.6]{GobelTrlifaj}. Moreover,
by \cite[Theorem 4.2.1]{GobelTrlifaj}, the left hand class of a
cotorsion pair generated by a set is deconstructible (the definition
will be precised later) and it has recently proved that
deconstructible classes are precovering (see \cite[Theorem
2.14]{SaorinStovicek} for a proof in exact categories and
\cite[Theorem 5.5]{Enochs12} for a proof in module categories), and
that deconstructible classes closed under products are preenveloping
\cite[Theorem 4.19]{SaorinStovicek}.

In this paper we are going to work with deconstructible classes. We
are going to give easier proofs of the aforementioned results concerning
deconstructible classes and approximations. Next we will use this
results to prove that, if $(\mathcal T, \mathcal S)$ is a Tor-pair and
$m$ a natural number, then $\mathcal T_m$ is always precovering, and
is preenveloping povided it is closed under direct products, i. e.,
they are satisfied the conditions of Corollary \ref{c:Definable}.

Given a class of right $R$-modules $\mathcal G$, a
\textit{$\mathcal G$-filtration} of a module $M$ is a continuous chain
of submodules of $M$, $(G_\alpha:\alpha < \kappa)$, where $\kappa$ is
a cardinal, such that $M=\bigcup_{\alpha < \kappa}G_\alpha$, $G_0=0$
and $\frac{G_{\alpha+1}}{G_\alpha} \in \mathcal G$ for each
$\alpha < \kappa$. We shall denote by $\Filt-\mathcal G$ the class of
all $\mathcal G$-filtered modules. We shall say that a class of
modules $\mathcal X$ is \textit{deconstructible} if there exists a set
of modules $\mathcal G$ such that $\mathcal X = \Filt-\mathcal G$.

We begin proving that a deconstructible class is precovering. Given a
class of right modules $\mathcal X$ and a module $M$, \textit{the trace of
  $\mathcal X$ in $M$} is the submodule
\begin{displaymath}
  t_{\mathcal X}(M) = \sum_{\substack{f \in \Hom(X,M)\\X \in
      \mathcal X}} \Img f
\end{displaymath}
The module $M$ is said to be \textit{generated by $\mathcal X$} if there exists
a family of modules in $\mathcal X$, $\{X_i:i \in I\}$, and an
epimorphism $\varphi:\bigoplus_{i \in I}X_i \rightarrow M$. We shall
denote by $\Gen (\mathcal X)$ the class of all modules generated by
$\mathcal X$. Recall that $M \in \Gen(\mathcal X)$ if and only if
$t_\mathcal{X}(M)=M$ \cite[Proposition 8.12]{AndersonFuller}.

\begin{lem}
  Let $\mathcal X$ be a class of right modules and $M$ a module. Suppose
  that $f:X \rightarrow M$ is a $\mathcal X$-precover of
  $M$. Then $\Img f = t_{\mathcal X}(M)$.
\end{lem}

\begin{proof}
  Clearly $\Img f \leq t_{\mathcal X}(M)$. The other inclusion follows
  form the fact that for each morphism $g:X' \rightarrow M$ with
  $X' \in \mathcal X$, there exists $h:X' \rightarrow X$ with
  $f h = g$ and, consequently, $\Img g \leq \Img f$.
\end{proof}

\begin{lem}\label{l:PrecoveringGen}
  Let $\mathcal X$ be a class of right modules. Then $\mathcal X$ is
  precovering if and only if each module in $\Gen (\mathcal X)$ has a
  $\mathcal X$-precover.
\end{lem}

\begin{proof}
  Supppose that every module in $\Gen(\mathcal X)$ has a
  $\mathcal X$-precover and let $M$ be any module. Since
  $t_{\mathcal X}(M) \in \Gen(\mathcal X)$ by \cite[Proposition
  8.12]{AndersonFuller}, there exists a $\mathcal X$-precover
  $f:X \rightarrow t_\mathcal{X}(M)$. We claim that $if$ is a
  $\mathcal X$-precover of $M$, where
  $i:t_{\mathcal X}(M) \rightarrow M$ is the inclusion: for any
  $g:X' \rightarrow M$ with $X' \in \mathcal X$, since
  $\Img g \leq t_{\mathcal X}(M)$, $g$ factors through
  $t_{\mathcal X}(M)$. Then $g=i \overline g$ for some
  $\overline g:X' \rightarrow t_{\mathcal X}(M)$. As $f$ is a
  $\mathcal X$-precover, there exists $h:X' \rightarrow F$ with
  $fh = \overline g$. Then $fhi=\overline g i = g$. This proves the
  claim.
\end{proof}

As we mentioned before, if $(\mathcal F,\mathcal C)$ is a cotorsion
pair generated by a set of modules $\mathcal G$, then $\mathcal F$ is
precovering and $\mathcal C$ is preenveloping. More
precisely, \cite[Theorem 3.2.1]{GobelTrlifaj} asserts that each module
has a $\mathcal C$-preenvelope with cokernel in $\Filt-\mathcal
G$. Using this result, Lemma \ref{l:PrecoveringGen} and the argument
in Salce Lemma \cite[Lemma 2.2.6]{GobelTrlifaj}, we can give an easy
proof to the fact that any deconstructible class is precovering. This
result was proved in \cite{SaorinStovicek} in exact categories and in
\cite{Enochs12} using module theory techniques.

\begin{thm}\label{t:DeconstructivePrecovering}
  Any deconstructible class of right modules is precovering.
\end{thm}

\begin{proof}
  Let $\mathcal G$ be any set of modules. In view of Lemma
  \ref{l:PrecoveringGen} we only have to see that each module in
  $\Gen(\Filt-\mathcal G)$ has a $(\Filt-\mathcal G)$-precover. Let
  $M \in \Gen(\Filt-\mathcal G)$ and take $f:F \rightarrow M$ an
  epimorphism with $F \in \Filt-\mathcal G$. By \cite[Theorem
  3.2.1]{GobelTrlifaj}, there exists a short exact sequence
  \begin{displaymath}
    \begin{tikzcd}
      0 \arrow{r} & \Ker f \arrow{r}{i} & P \arrow{r} & N \arrow{r} &
      0
    \end{tikzcd}
  \end{displaymath}
  with $P \in \mathcal G^\perp$ and $N \in \Filt-\mathcal G$. If we
  compute the pushout $i$ and the inclusion $j:\Ker f \rightarrow M$,
  we get the following commutative diagram with exact rows and colummns:
  \begin{displaymath}
    \begin{tikzcd}
      & 0 \arrow{d} & 0 \arrow{d} & & \\
      0 \arrow{r} & \Ker f \arrow{r}{j} \arrow{d}{i} & F \arrow{r}{f}
      \arrow{d} & M \arrow{r} \arrow[equal]{d} & 0\\
      0 \arrow{r} & P \arrow{r} \arrow{d} & Q \arrow{r}{g} \arrow{d} &
      M
      \arrow{r} & 0\\
      & N \arrow[equal]{r} \arrow{d} & N \arrow{d} & & \\
      & 0 & 0 & & \\
    \end{tikzcd}
  \end{displaymath}
  Since $F$ and $N$ belong to $\Filt-\mathcal G$, then so does
  $Q$. But $P \in \mathcal G^\perp$, so that, by Ekolf Lemma
  \cite[Lemma 3.1.2]{GobelTrlifaj},
  $P \in \left(\Filt-\mathcal G\right)^\perp$ as well (note that if
  $X \in \Filt-\mathcal G$, then $X$ is ${^\perp}P$-filtered so that
  $X \in {^\perp}P$; i. e. $P \in X^\perp$). Consequently, $g$ is a
  $(\Filt-\mathcal G)$-precover of $M$.
\end{proof}

\begin{rem}
  Given a class of modules $\mathcal X$ and a module $M$, a special
  $\mathcal X$-precover of $M$ is a morphism $f:X \rightarrow M$
  with $\Img f = t_{\mathcal X}(M)$ and $\Ker f \in \mathcal X^\perp$
  (note that we are not imposing that $f$ is epic [which cannot be if
  $\mathcal X$ is not generating!] as it is done in the classical
  definition of special preenvelopes \cite[Definition
  2.1.12]{GobelTrlifaj}). With this definition, the preceeding theorem
  actually proves that any deconstructible
  class is special precovering.
\end{rem}

Now we prove that a deconstructible class closed under products is
preenveloping. In order to do this, we are going to use the following
technical property which is employed in \cite[Corollary
6.2.2]{EnochsJenda}. This property is related with the cardinality
condition defined in \cite[Definition 1.1]{HolmJorgensen}.

\begin{defn}
  Let $\mathcal X$ be a class of right modules and $\lambda$ an infinite
  cardinal. We say that $\mathcal X$ satisfies the property
  $\mathbb P_\lambda$ if there exists an infinite cardinal
  $\kappa_\lambda$ with the following property:
  \begin{displaymath}
    X \in \mathcal X, S \leq X, |S| \leq \lambda \Rightarrow \exists Y
    \leq X \textrm{ with }S \leq Y, Y \in \mathcal X \textrm{ and }|Y|
    \leq \kappa_\lambda
  \end{displaymath}
\end{defn}

We establish the relationship between this property and the
existence preenvelopes, which was proved in
\cite[Corollary 6.2.2]{EnochsJenda} (see \cite[Proposition
1.2]{HolmJorgensen} too).

\begin{prop}\label{p:PropertyPAndPrecovers}
  Let $\mathcal X$ be a class of right $R$-modules and $\lambda$ an
  infinite cardinal. If $\mathcal X$ is closed under products and
  satisfies the property $\mathbb P_\lambda$, then each module $M$
  with cardinality less than or equal to $\lambda$ has a
  $\mathcal X$-preenvelope.
\end{prop}

The idea of the proof of this theorem is to take, given a module $M$
with cardinality less than or equal to $\lambda$, a representing set
$\mathcal G$ of the class of all modules in $\mathcal X$ with
cardinality less than or equal to $\kappa_\lambda$. Then the canonical
morphism from $M$ to $\prod_{G \in \mathcal G}G^{\Hom_R(M,G)}$ is
trivially a $\mathcal X$-preenvelope as a consequence of the property
$\mathbb P_\lambda$.

\begin{rem}
  We could consider the dual property of $\mathbb P_\lambda$: we say that
  $\mathcal X$ satisfies the property $\mathbb Q_\lambda$ if there
  exists an infinite cardinal $\kappa_\lambda$ satisfying
  \begin{displaymath}
    X \in \mathcal X, S \leq X, \left|\frac{X}{S}\right| \leq \lambda
    \Rightarrow \exists S'
    \leq S \textrm{ with } \frac{X}{S'} \in \mathcal X \textrm{ and }\left|\frac{X}{S'}\right|
    \leq \kappa_\lambda
  \end{displaymath}
  The property $\mathbb Q_\lambda$ is related with the existence of
  precovers: if $\mathcal X$ is closed under direct sums and satisfies
  the property $\mathbb Q_\lambda$, then each module $M$ with
  cardinality less than or equal to $\lambda$ has a
  $\mathcal X$-precover. $\mathbb Q_\lambda$ is related with the
  co-cardinality condition defined in \cite[Definition
  1.1]{HolmJorgensen}. 
\end{rem}

One useful tool to deal with filtrations is Hill Lemma
\cite[Theorem 4.2.6]{GobelTrlifaj}. Roughly speaking, it states that a
filtration of a module can be enlarged to a class of submodules with
certain properties. Recall that a cardinal $\kappa$ is \textit{regular} if it is not the union of less than $\kappa$ sets
with cardinality less than $\kappa$. Recall that given an infinite regular cardinal
$\kappa$, a module $M$ is $\kappa$-presented if it there exists a
presentation of $M$ with $\kappa$ generators and $\kappa$ relations.

\begin{thm}
  Let $\kappa$ be an infinite regular cardinal and $\mathcal G$ a set
  of $< \kappa$-presented right modules. Let $M$ be a module with a
  $\mathcal G$-filtration, $M=\bigcup_{\alpha < \mu}M_\alpha$. Then
  there is a family $\mathcal H$ of submodules of $M$ such that:
  \begin{enumerate}
  \item[(H1)] $M_\alpha \in \mathcal H$ for each $\alpha < \mu$.

  \item[(H2)] $\mathcal H$ is closed under arbitrary sums and
    intersections.

  \item[(H3)] Let $N, P \in \mathcal H$ such that $N \leq P$. Then
    $\frac{P}{N}$ is filtered by modules in
    $\left\{\frac{M_{\alpha+1}}{M_\alpha}:\alpha < \mu\right\}$.

  \item[(H4)] Let $N \in \mathcal H$ and $X$ a subset of $M$ of
    cardinality smaller than $\kappa$. Then there is a
    $P \in \mathcal H$ such that $N \cup X \leq P$ and $\frac{P}{N}$
    is $< \kappa$-presented.
  \end{enumerate}
\end{thm}

Using Hill lemma, we prove that a deconstructible class satisfies $\mathbb Q_\lambda$.

\begin{prop}
  Let $\mathcal X$ be a deconstructible class of right modules. Then $\mathcal X$
  satisfies $\mathbb Q_\lambda$ for each
  infinite cardinal $\lambda$.
\end{prop}

\begin{proof}
  Since $\mathcal X$ is deconstructible, there exists a set
  $\mathcal G$ such that $\mathcal X = \mathcal Filt-\mathcal G$. Let
  $\kappa$ be an infinite regular cardinal such that each module in
  $\mathcal G$ is $< \kappa$-presented.

  Let $\lambda$ be an infinite cardinal and set
  $\kappa_\lambda := \max\{\lambda,\kappa\}$. Let $X$ be a module in
  $\mathcal X$ and $S \leq X$ with $|X| \leq \lambda$. Now denote by
  $\mathcal H$ the family of submodules of $X$ given by the Hill
  Lemma.

  If $\lambda < \kappa$ then $S$ is contained in a
  $< \kappa$-presented submodule of $M$ by (H4) which belongs to
  $\mathcal X$ by (H3). Then
  $\mathbb P_\lambda$ is satisfied since $\kappa_\lambda = \kappa$ in
  this case.

  Suppose that $\lambda \geq \kappa$ so that
  $\kappa_\lambda = \lambda$. Write $S = \{s_\alpha:\alpha < \mu\}$
  for some cardinal $\mu \leq \lambda$. Applying recursively (H4) and
  (H2) it is easy to construct a continuous chain of submodules of
  $M$, $\{S_\alpha:\alpha < \mu\}$, with $S_\alpha \in \mathcal H$,
  $\frac{S_{\alpha+1}}{S_\alpha}$ $<\kappa$-presented and $x_\alpha
  \in S_{\alpha+1}$ for each $\alpha < \mu$. Now, by (H3),
  $\frac{S_{\alpha+1}}{S_\alpha} \in \mathcal X$ so that
    $Y=\bigcup_{\alpha < \mu}S_\alpha$ is $\mathcal X$-filtered;
    by \cite[Corollary 2.11]{SaorinStovicek}, $Y$ belongs to $\mathcal
    X$. Moreover $Y$ has cardinality less than or equal to $\lambda =
    \kappa_\lambda$. Then $\mathcal X$ satisfies $\mathbb P_\lambda$.
\end{proof}

As a consequence of this result and Proposition
\ref{p:PropertyPAndPrecovers} we inmediately get that a
deconstructible class closed under product is preenveloping. This
result was proved in \cite[Theorem 4.19]{SaorinStovicek} for exact
categories. 

\begin{thm}\label{t:DeconstructiblePreenveloping}
  Each deconstructible
  class of right modules closed under products is preenveloping.
\end{thm}

Now we apply these results to Tor-pairs

\begin{cor}\label{c:TApproximations}
  Let $(\mathcal T, \mathcal S)$ be a Tor-pair and $m$ a natural
  number. Then:
  \begin{enumerate}
  \item The class $\mathcal T_m$ is precovering.

  \item If $R$ satisfies the equivalent condition of Corollary
    \ref{c:Definable}, $\mathcal T_m$ is preenveloping.
  \end{enumerate}
\end{cor}

\begin{proof}
  By \cite[Theorem 8]{EklofTrlifaj}, $\mathcal T$ is a deconstructible
  class, since $\mathcal T = {^\perp}\mathcal I$, where $\mathcal I =
  \{S^c:S \in \mathcal S\}$ by \cite[Lemma 2.2.3]{GobelTrlifaj} ($S^c$
  being the caracter module of $S$). By \cite[Proposition
  3.2]{Izurdiaga}, $\mathcal T_m$ is deconstructible as well. Then (1)
  follows from Theorem \ref{t:DeconstructivePrecovering} and (2) from
  Theorem \ref{t:DeconstructiblePreenveloping}.
\end{proof}

\bibliographystyle{alpha} \bibliography{references}

\end{document}